\documentclass[12pt, reqno]{amsart}
\usepackage{amsmath,times,epsfig,amssymb,amsbsy,amscd,amsfonts,amstext,color,bm}

\usepackage{bm}

\usepackage{tikz}

\numberwithin{equation}{section}

\theoremstyle{plain}
\newtheorem{theorem}{Theorem}[section]

\newtheorem{proposition}[theorem]{Proposition}

\newtheorem{problem}[theorem]{Problem}

\theoremstyle{definition}
\newtheorem{definition}[theorem]{Definition}
\newtheorem{example}[theorem]{Example}

\theoremstyle{remark}
\newtheorem{remark}[theorem]{Remark}

\newcommand{\A}{\mathcal{A}}

\newcommand{\Z}{\mathbb{Z}}

\newcommand{\R}{\mathbb{R}}
\newcommand{\C}{\mathbb{C}}

\newcommand{\scF}{\mathcal{F}}

\newcommand{\scL}{\mathcal{L}}

\newcommand{\CG}{\operatorname{\mathcal{Z}}}
\newcommand{\cl}{\gamma}
\newcommand{\Uqsl}{U_q(\frak{sl}_2)}

\newcommand{\ch}{\operatorname{ch}}

\newcommand{\rank}{\operatorname{rank}}
\newcommand{\codim}{\operatorname{codim}}
\newcommand{\End}{\operatorname{End}}

\newcommand{\Ker}{\operatorname{Ker}}
\renewcommand{\Im}{\operatorname{Im}}

\begin{document}

\title[$q$-Aomoto complex]{$q$-deformation of Aomoto complex}

\begin{abstract}
A degree one element of the Orlik-Solomon algebra of a 
hyperplane arrangement defines a cochain complex 
known as the Aomoto complex. The Aomoto complex 
can be considerd as the ``linear approximation'' of the 
twisted cochain complex with coefficients in 
a complex rank one local system. 

In this paper, we discuss $q$-deformations of the Aomoto complex. 
The $q$-deformation is defined by replacing the entries of 
representation matrices of the coboundary maps with their 
$q$-analogues. 
While the resulting maps do not generally define cochain complexes, 
for certain special basis derived from 
real structures, the $q$-deformation becomes again 
a cochain complex. 
Moreover, it exhibits universality in the sense that 
any specialization of $q$ to a complex number yields 
the cochain complex computing the corresponding local system 
cohomology group. 
\end{abstract}

\author{Masahiko Yoshinaga}
\address{Masahiko Yoshinaga, 
%Department of Mathematics, Faculty of Science, Hokkaido University, 
%Kita 10, Nishi 8, Kita-Ku, Sapporo 060-0810, Japan.
Osaka University}
\email{yoshinaga@math.sci.osaka-u.ac.jp}

\thanks{This work was partially supported by 
JSPS KAKENHI Grant Numbers JP23H00081}

%Primary 14C21,
%14F99, 32S22 ; Secondary 14E05, 14H50.}

\dedicatory{Dedicated to 
Prof. Alex Suciu and Prof. Lauren\c{t}iu P\u{a}unescu \\
on the occasion of their 70th birthday}

\date{\today}

\maketitle

%\tableofcontents

\section{Introduction}
\label{sec:intro}

Let $\A$ be a hyperplane arrangement in $\C^\ell$. 
We denote the complement by 
$M(\A)=\C^\ell\smallsetminus\bigcup_{H\in\A}H$. 
These spaces hold significant importance in  
topology and various areas of mathematics. 
One of the central themes in 
the study of hyperplane arrangements is the relationship 
between the topological structures of $M(\A)$ and 
the combinatorial structures of $\A$. 

The cohomology ring of $M(\A)$ has an expression as 
the Orlik-Solomon algebra (\cite{orl-ter}) which is 
defined in a purely combinatorial way by the intersection poset 
$L(\A)$ of $\A$. In some cases, the local system cohomology 
groups and Betti numbers of the covering spaces are also 
described in combinatorial way. Nevertheless, 
these problems remain still open in general 
(see \S \ref{sec:pbm} for further details). 

In this paper, we focus on complex rank $1$ local systems. 
One of the key notions in exploring local systems is the 
\emph{Aomoto complex} $(A_{\Z}^\bullet(\A), \omega)$, 
originally introduced in the context of higher-dimensional 
generalizations of hypergeometric integrals \cite{ao-kita}. 
The Aomoto complex is defined by using 
the Orlik-Solomon algebra, hence it is described by 
the intersection poset. It is known that the cohomology of the 
Aomoto complex approximates 
the local system cohomology groups, namely, it can be 
considered as the ``linearlized version'' of local system 
cochain complex. 

The purpose of this paper is to appropriately deform 
the Aomoto complex in a such a way that it 
recovers arbitrary local system cohomology groups 
as its specializations (``universality''). 
The idea is to construct the universal cochain complex as a 
``$q$-deformation of the Aomoto complex'' (Figure \ref{fig:qdef}) 
which: 
\begin{itemize}
\item 
is a coshain complex over the ring of 
Laurent polynomials invariant under the involution 
$\CG=\Z[q^{1/2}, q^{-1/2}]^{\langle\iota\rangle}$ 
(see \S \ref{sec:q} for details), 
\item 
reverts to the Aomoto complex as $q\to 1$, 
\item 
can compute the local system cohomology group by specializing 
$q$ to a complex number $q_0\neq 0, 1$. 
\end{itemize}

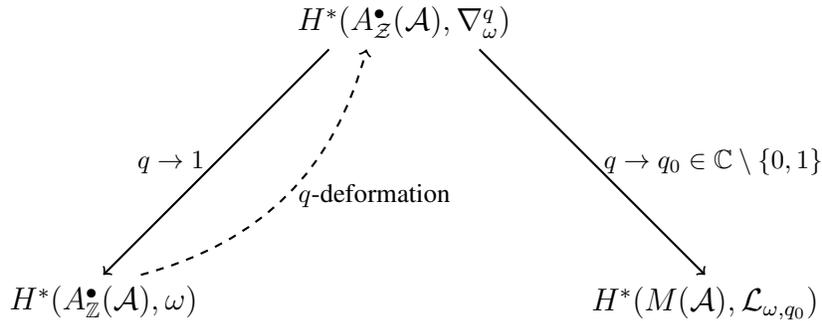
\begin{figure}[htbp]
\centering
\begin{tikzpicture}%[scale=0.8]

%格子
%\draw [help lines] (0,0) grid (12,6);%(0,0)から(10,4)までの"細線の方眼"

\draw (0,0) node[below] {$H^*(A_{\Z}^\bullet(\A), \omega)$};
\draw (4,3) node[above] {$H^*(A_{\CG}^\bullet(\A), \nabla_\omega^q)$};
\draw[thick, <-] (0,0)-- node[left] {\footnotesize $q\to 1$} (3,3);
\draw[thick, dashed, ->] (0.5,0) to [out=15, in=255] node[right] {\footnotesize $q$-deformation} (3.5,3);

\draw[thick, ->] (5,3)-- node[right] {\footnotesize $q\to q_0\in\C\setminus\{0, 1\}$} (8,0) node[below] {$H^*(M(\A), \scL_{\omega, q_0})$};

%\draw[thick, black] (0,0) -- +(8,4) node[right]{$H_3$}; 
%\draw[thick, black] (0,2) -- +(8,4) node[right]{$H_4$}; 
%\draw[thick, black] (0,4) -- +(8,-4) node[right]{$H_1$}; 

\end{tikzpicture}
\caption{$q$-deformation and specializations of Aomoto complex} 
\label{fig:qdef}
\end{figure}

The paper is organized as follows. In \S \ref{sec:q}, we recall the 
notion of $q$-integers and the relations among them. 
In \S \ref{sec:pbm}, we recall basic facts on hyperplane 
arrangements and motivating problems related to local system 
cohomology groups and covering spaces. 
In \S \ref{sec:main}, we consider $q$-deformations of 
the Aomoto complex. First we see that a naive 
$q$-deformation does not maintain being a cochain complex 
in general. 
Therefore, we need to choose carefully the 
basis of the Aomoto complex. We observe that, for some bases, 
the canonical $q$-deformation becomes again a cochain complex. 
Then we formulate the main result that complexified real 
line arrangements have an ideal basis such that the canonical 
$q$-deformation becomes the universal cochain complex. 
The proof relies on previous works. 
In \S \ref{sec:ex}, we exhibit examples. We compute local system 
cohomology groups for certain local systems on the deleted 
$B_3$-arrangement. We also discuss further problems. 

\section{$q$-integers}

\label{sec:q}

\subsection{$q$-analogue of integers}
The Laurent polynomial in the variable $q^{1/2}$ 
\[
[n]_q:=\frac{q^{\frac{n}{2}}-q^{-\frac{n}{2}}}{q^{\frac{1}{2}}-q^{-\frac{1}{2}}}=
q^{\frac{n-1}{2}}+
q^{\frac{n-3}{2}}+\cdots+
q^{\frac{1-n}{2}}
\]
is called the \emph{$q$-analogue of an integer $n\in\Z$} 
(or the \emph{canonical $q$-deformation} of $n$). 
(For instance, $[0]_q=0, [1]_q=1, [2]_q=q^{\frac{1}{2}}+q^{-\frac{1}{2}}, 
[3]_q=q^{1}+1+q^{-1}, 
[4]_q=q^{\frac{3}{2}}+q^{\frac{1}{2}}+q^{-\frac{1}{2}}+q^{-\frac{3}{2}}$). 
Note that the assignment $n\longmapsto [n]_q$ 
is not a homomorphism. 

Define the involution $\iota:\Z[q^{1/2}, q^{-1/2}]\to
\Z[q^{1/2}, q^{-1/2}]$ by $\iota(q^{1/2})=q^{-1/2}$. 
Clearly $[n]_q$ is $\iota$-invariant, hence 
$[n]_q\in\Z[q^{1/2}, q^{-1/2}]^{\langle\iota\rangle}$. 
Conversely, we have 
the following. 
%The following is straightforward. 

\begin{proposition}
\label{prop:inv}
The invariant submodule 
$\Z[q^{1/2}, q^{-1/2}]^{\langle\iota\rangle}$ is freely generated by 
$[n]_q$ ($n\in\Z_{>0}$) as $\Z$-module, hence 
\begin{equation}
\label{eq:gen}
\Z[q^{1/2}, q^{-1/2}]^{\langle\iota\rangle}=
\bigoplus_{n=1}^\infty\Z\cdot [n]_q. 
\end{equation}
\end{proposition}

\subsection{Clebsch-Gordan relation}

\begin{proposition}
The $q$-analogues of integers satisfy the following. 
\label{prop:qrel}
\begin{itemize}
\item[(1)] 
For any $n\in\Z$, 
\begin{equation}
\label{eq:minus}
[-n]_q=-[n]_q.
\end{equation} 
\item[(2)] 
Let $m, n\in\Z$. Suppose $n> 0$. Then 
\begin{equation}
\label{eq:cg}
[m]_q[n]_q=[m+n-1]_q+[m+n-3]_q+\cdots +[m-n+1]_q. 
\end{equation}
\end{itemize}
\end{proposition}
The second assertion is called the Clebsch-Gordan relation. 
Since it is elementary, we omit the proof. 

\begin{remark}
\label{rem:cg}
Let us briefly describe the relation between the formula 
(\ref{eq:cg}) and representations of the quantum group 
$\Uqsl$. First, $\Uqsl$ is defined, as the $\C$-algebra, generated 
by the four variables $E, F, K, K^{-1}$ with the relations 
(\cite[Definition VI. 1.1]{kas-qg}) 
\[
\begin{split}
&KK^{-1}=K^{-1}K=1,\\
&KEK^{-1}=q^2 E,\ KFK^{-1}=q^{-2}F,\\
&EF-FE=\frac{K-K^{-1}}{q-q^{-1}}. 
\end{split}
\]
There exists an $(n+1)$-dimensional irreducible representation 
$\rho_n:\Uqsl\longrightarrow\End(V_n)$ which is called the 
highest weight representation with the highest weight $q^n$ 
(\cite[Theorem VI. 3.5]{kas-qg}). 
The trace of $\rho_n(K)$ can be easily computed as 
(\cite[page 129]{kas-qg})
\[
\operatorname{Tr}(\rho_n(K))=q^n+q^{n-2}+\dots+q^{-n}=[n+1]_{q^2}. 
\]
The tensor products of these representations are not irreducible. 
The irreducible decomposition is described by the following 
Clebsch-Gordan formula (\cite[Theorem VII. 7.2]{kas-qg}) 
\begin{equation}
\label{eq:cgdec}
V_n\otimes V_m\simeq 
V_{n+m}\oplus V_{n+m-2}\oplus\cdots\oplus V_{n-m}, 
\end{equation}
for $n\geq m\geq 0$. By comparing the trace of $K$ for both sides, 
we obtain the formula (\ref{eq:cg}). 
\end{remark}

Next, we prove that relations among $q$-integers are generated 
by the relations (\ref{eq:minus}) and (\ref{eq:cg}) together with 
$[0]_q=0, [1]_q=1$. More precisely, we have the following. 

\begin{proposition}
\label{prop:ring}
Let $\Z[x_n; n\in\Z]$ 
be the polynomial ring generated by $x_n (n\in\Z)$, 
and let $I$ be the ideal generated by the following elements. 
\[
\begin{cases}
(i) \ x_0, \\
(ii) \ x_1-1, \\
(iii) \ x_n+x_{-n}, \ \ (n\in\Z)\\
(iv) \ x_m x_n-(x_{m+n-1}+x_{m+n-3}+\dots+x_{m-n+1}), \ \ 
(m, n\in\Z, m\geq n>0). 
\end{cases}
\]
Then, the assignment $\varphi(x_n)=[n]_q$ induces an isomorphism 
of $\Z$-algebras 
\[
\varphi: 
\Z[x_n; n\in\Z]/I \stackrel{\simeq}{\longrightarrow} 
\Z[q^{1/2}, q^{-1/2}]^{\langle\iota\rangle}. 
\]
\end{proposition}
\begin{proof}
By Proposition \ref{prop:qrel} and Proposition \ref{prop:inv}, 
the map $\varphi: \Z[x_n; n\in\Z]/I \to 
\Z[q^{1/2}, q^{-1/2}]^{\langle\iota\rangle}$ is a well-defined 
surjective homomorphism. 

It remains to show that $\varphi$ is injective. 
Suppose $g(x_i)\in\Ker(\varphi)$. Using the generators of 
the ideal $I$, we may assume that $g(x_i)$ is of the form 
\[
g(x_i)=\sum_{i\geq 2} a_ix_i+b, 
\]
$a_i, b\in\Z$ ($a_i$ is zero except for finitely many $i$). 
By the assumption, $\sum_{i\geq 2}a_i[i]_q+b=0$ in 
$\Z[q^{1/2}, q^{-1/2}]^{\langle\iota\rangle}$. Since $[n]_q$ ($n\geq 1$) 
are linearly independent over $\Z$, we have $a_i=b=0$. 
\end{proof}
Denote the algebra in Proposition \ref{prop:ring} by 
\begin{equation}
\CG:=\Z[x_n; n\in\Z]/I 
\simeq
\Z[q^{1/2}, q^{-1/2}]^{\langle\iota\rangle}. 
\end{equation}
There is a natural ring homomorphism 
\begin{equation}
\cl:\CG\longrightarrow\Z 
\end{equation}
defined by $x_n\longmapsto n$, 
which is equivalent to $q\longmapsto 1$. 

\begin{definition}
\begin{itemize}
\item[$(1)$] 
A \emph{$q$-deformation} of an integer $n\in\Z$ is an 
element $\widetilde{n}\in\CG$ such that 
$\cl(\widetilde{n})=n$. 
\item[$(2)$] 
Among others, the $q$-deformation 
$n\longmapsto n\cdot[1]_q$ is called the \emph{trivial $q$-deformation}, 
and $n\longmapsto [n]_q$ is called the 
\emph{canonical $q$-deformation}. 
\item[$(3)$] 
A $q$-deformation of an integer matrix $A=(a_{ij})_{ij}\in M_{m,n}(\Z)$ 
is a matrix $\widetilde{A}=(\widetilde{a_{ij}})_{ij}\in M_{m,n}(\CG)$ 
such that $\cl(\widetilde{a_{ij}})=a_{ij}$. 
\item[$(4)$] 
A $q$-deformation of a homomorphism $f:M_1\longrightarrow M_2$ of 
$\Z$-modules $M_i (i=1,2)$ is a homomorphsim 
$\widetilde{f}:\widetilde{M_1}\longrightarrow \widetilde{M_2}$ of 
$\CG$-modules together with isomorphisms 
$\widetilde{M_i}\otimes_{\CG} \Z\simeq M_i (i=1,2)$ which commute 
the following diagram. 
\[
\begin{CD}
\widetilde{M_1}\otimes_{\CG} \Z @>\widetilde{f}\otimes_{\CG}\Z >>
\widetilde{M_2}\otimes_{\CG} \Z\\
@V\simeq VV @VV\simeq V\\
M_1 @>>f> M_2
\end{CD}
\]
\item[(5)] 
Let $f: \Z^m\longrightarrow\Z^n$ be a homomorphism of 
free abelian groups represented by a matrix $A=(a_{ij})_{ij}$. 
The \emph{canonical deformation} of $f$ is 
$\widetilde{f}:\CG^m\longrightarrow\CG^n$ defined by 
the matrix $\widetilde{A}=([a_{ij}]_q)_{ij}$. 
\end{itemize}
\end{definition}

\begin{remark}
There are many $q$-deformations other than the trivial and 
the canonical deformations, e.g., 
$n\longmapsto [n-2]_q+[2]_q$. 
\end{remark}

\section{Problems related to Local system cohomology groups}

\label{sec:pbm}

\subsection{Hyperplane arrangements and Aomoto complex}

\label{subsec:aom}

Let $V=\C^\ell$. An arrangement of hyperplanes 
is a set $\A=\{H_1, \dots, H_n\}$ consisting of finitely many 
affine hyperplanes in $V$. We denote its complement by 
$M(\A)=\C^\ell\smallsetminus\bigcup_{H\in\A}H$. 
Let $\alpha_H$ be a defining linear equation of the hyperplane $H$. 
Then the first cohomology group $H^1(M(\A), \Z)$ is a free 
abelian group generated by 
\[
e_i=\frac{1}{2\pi\sqrt{-1}} d\log\alpha_{H_i}\in H^1(M(\A), \Z). 
\]
Dually, the first homology group $H_1(M(\A), \Z)$ is a free abelian 
group generated by the meridian cycles $\gamma_i$ of 
each hyperplane $H_i$. 

The cohomology ring $H^*(M(\A), \Z)$ is generated by $e_1, \dots, e_n$. 
Furthermore, the following presentation (by Orlik and Solomon) is 
known \cite{orl-ter}. Let $E=\bigoplus_{i=1}^n\Z e_i$ be the free abelian 
group generated by $e_1, \dots, e_n$ and $\wedge E$ be its exterior 
algebra. There are two types of algebraic relations 
among $e_1, \dots, e_n$ in $H^*(M(\A), \Z)$. The first one is 
\begin{equation}
e_{i_1}e_{i_2}\cdots e_{i_p}=0,  
\end{equation}
for $H_{i_1}\cap\dots\cap H_{i_p}=\emptyset$. 
The second one is 
\begin{equation}
\partial(e_{i_1}e_{i_2}\dots e_{i_p}):=
\sum_{\alpha=1}^p
(-1)^{\alpha-1}
e_{i_1}\dots\widehat{e_{i_\alpha}}\dots e_{i_p}=0, 
\end{equation}
for dependent subset $\{H_{i_1}, \dots, H_{i_p}\}$, namely, 
the subset of $\A$ satisfying 
($H_{i_1}\cap\dots\cap H_{i_p}\neq\emptyset$) 
$\codim H_{i_1}\cap\dots\cap H_{i_p}<p$. 

Let $I_\A$ be the graded ideal of the exterior algebra $\wedge E$ 
generated by 
\[
\begin{cases}
e_{i_1}e_{i_2}\dots e_{i_p} & \mbox{ for } H_{i_1}\cap\dots\cap H_{i_p}=\emptyset, \\
\partial(e_{i_1}e_{i_2}\dots e_{i_p})& \mbox{ for }
\codim H_{i_1}\cap\dots\cap H_{i_p}<p. 
\end{cases}
\]
Denote by $A_\Z^\bullet(\A)=E/I_{\A}$ the quotient, 
which is a graded algebra 
called the \emph{Orlik-Solomon algebra} of $\A$. 
Note that the Orlik-Solomon algebra is defined by only using 
the poset structure of the non-empty intersections 
\[
L(\A)=\{ \bigcap_{H\in\mathcal{B}}H\neq\emptyset\mid 
\mathcal{B}\subset\A\}. 
\]
The Orlik-Solomon algebra is known to be isomorphic to 
the cohomology ring, namely, we have the algebra isomorphism 
\[
H^\bullet (M(\A), \Z)\simeq
A_\Z^\bullet(\A). 
\]
Let $\omega=\sum_{i=1}^n a_ie_i\in A_\Z^1(\A)$. Then, since 
$\omega\wedge\omega=0$, 
\[
\begin{CD}
(A_\Z^\bullet(\A), \omega):\ \  \cdots 
@>\omega\wedge>> A_\Z^k(\A)
@>\omega\wedge>> A_\Z^{k+1}(\A)
@>\omega\wedge>> \cdots
\end{CD}
\]
forms a cochain complex, which is called the \emph{Aomoto complex}. 

\subsection{Local system cohomology group}
\label{subsec:locsys}

Let 
\[
\rho:\pi_1(M(\A))\longrightarrow\C^\times
\]
be a group homomorphism. 
Since $\C^\times$ is an abelian group, 
the map $\rho$ factors through the abelianization $H_1(M(\A), \Z)$. 
Hence, $\rho$ is determined by $q_i=\rho(\gamma_i)\in\C^\times$ 
($i=1, \dots, n$), where $\gamma_i$ is the meridian cycle around 
the hyperplane $H_i$. 
Denote the associated complex rank one local system by 
$\scL_\rho$. The computation of the twisted cohomology group 
$H^k(M(\A), \scL_\rho)$ is one of the central problems in 
the topology of hyperplane arrangements. The following problem is 
still open. 
\begin{problem}
\label{pbm:loc}
Can one describe $\dim H^k(M(\A), \scL_\rho)$ in terms of 
the intersection poset $L(\A)$ of $\A$ and 
$(q_1, \dots, q_n)\in(\C^\times)^n$? 
\end{problem}
Roughly speaking, if $\rho$ is close to the trivial local system 
(equivalently, $(q_1, \dots, q_n)$ is close to $(1, 1, \dots, 1)$ 
in $(\C^\times)^n$), 
$H^k(M(\A), \scL_\rho)$ can be described combinatorially. 
More precisely, suppose $q_i=\exp(2\pi\sqrt{-1}\lambda_i)$, and 
let 
$\omega_\lambda=\lambda_1e_1+\dots+\lambda_ne_n\in A_\C^1(\A)$. 
Then there 
exists an open subset $(0, \dots, 0)\in U\subseteq\C^n$ such that 
if $\lambda=(\lambda_1, \dots, \lambda_n)\in U$, the Aomoto complex 
$(A_\C^\bullet(\A), \omega_\lambda)$ is quasi-isomorphic to the 
twisted de Rham complex (\cite{esv, stv}). In particular, 
\[
H^k(M(\A), \scL_\rho)\simeq H^k(A_\C^\bullet(\A), \omega_\lambda), 
\]
which is also known as the tangent cone theorem \cite{suc-tan}. 

However, it is also known that the Aomoto complex 
$(A_\C^\bullet(\A), \omega_\lambda)$ with complex coefficients 
can not describe even 
$H^1(M(\A), \scL_\rho)$ in general. In fact, the characteristic 
variety 
\[
V_1(\A)=\{(q_1, \dots, q_n)\in(\C^\times)^n\mid 
\dim H^1(M(\A), \scL_\rho)\geq 1\}
\]
may have tosion translated components that can not be 
detected by the Aomoto complex 
$(A_\C^\bullet(\A), \omega_\lambda)$ (\cite{suc-trans}). 

\subsection{Motivating problems on covering spaces}
\label{subsec:motiv}

One of the motivations to study twisted cohomology groups 
is the topology of covering spaces. 
Although the following problems 
are subcases of Problem \ref{pbm:loc}, it is worth 
noting separately by their importance in relations to 
the topology of covering spaces. 
\begin{problem}
\label{pbm:2tor}
Let $\rho:\pi_1(M(\A))\longrightarrow\{\pm1\}$ 
be a surjective homomorphism. 
Can one describe $\dim H^k(M(\A), \scL_\rho)$ in terms of 
$L(\A)$ and $\rho$? 
\end{problem}
Problem \ref{pbm:2tor} is related to double coverings as follows. 
The index $2$ subgroup $\Ker(\rho)\subset\pi_1(M(\A))$ 
determines a double covering $p:M(\A)^\rho\longrightarrow M(\A)$. 
It is known that 
\[
H^k(M(\A)^\rho, \C)\simeq 
H^k(M(\A), \C)\oplus H^k(M(\A), \scL_\rho). 
\]
Hence Problem \ref{pbm:2tor} is equivalent to describe the 
Betti numbers of the double covering in terms of combinatorial 
information. Related topics are recently actively studied by 
several researchers 
\cite{isy, liu-liu, lmw, liu-xie, suc-boc, sug-cdo, yos-dou}. 

\begin{problem}
\label{pbm:mil}
Define $\rho:\pi_1(M(\A))\longrightarrow\C^\times$ by 
$\rho(\gamma_i)=\exp\left(\frac{2\pi\sqrt{-1}}{n+1}\right)$. 
Can one describe $\dim H^k(M(\A), \scL_{\rho^{\otimes i}})$ 
($i=0, 1, \dots, n$) in terms of $L(\A)$? 
\end{problem}
The kernel $\Ker(\rho)$ is a subgroup of $\pi_1(M(\A))$ of 
index $(n+1)$, and determines a $\Z_{n+1}$-cyclic covering 
$p:F\longrightarrow M(\A)$. The space $F$ is known to be 
homeomorphic to the Milnor fiber of the coning of $\A$. 
It is known that the cohomology group of the Milnor fiber 
is decomposed as 
\[
H^k(F, \C)\simeq\bigoplus_{i=0}^n H^k(M(\A), \scL_{\rho^{\otimes i}}), 
\]
which is equivalent to the decomposition of the cohomology group 
into the monodromy eigenspaces. 
Hence Problem \ref{pbm:mil} is equivalent to the combinatorial 
description of the Betti numbers 
(more precisely, the dimension of monodromy eigenspaces) 
of the Milnor fiber of (a central) hyperplane arrangement. 
See \cite{suc-mil} for more on Milnor fibers. 

Recently, Papadima and Suciu \cite{ps-sp, ps-modular}
discovered that Problem \ref{pbm:mil} is deeply related to the 
Aomoto complex 
$(A_{\mathbb{F}_q}^{\bullet} (\A), \omega=e_1+\dots+e_n)$ 
over the finite field $\mathbb{F}_q$, where $q$ is a prime power. 
Actually, if $3|(n+1)$, 
and the multiplicities of the intersections of 
corresponding projective lines are not contained in 
$\{6, 9, 12, \dots, 3n, \dots\}$, 
then, the dimension of the nontrivial part is 
\[
\sum_{i=1}^n \dim H^1(M(\A), \scL_{\rho^{\otimes i}})=2\cdot
\dim_{\mathbb{F}_3}H^1(A_{\mathbb{F}_3}^{\bullet} (\A), e_1+\dots+e_n). 
\]
However, the rank of the cohomology group of the Aomoto complex 
over finite fields generally provides only an upper bound 
(\cite{yos-dou}). Instead, in the next section, 
we will investigate the $q$-deformation 
of the integral Aomoto complex $A_\Z^\bullet(\A)$. 

\section{Main results}

\label{sec:main}

\subsection{$q$-deformable universal basis of Aomoto complex}

\label{subsec:qaom}

Let $\A=\{H_1, \dots, H_n\}$ be a line arrangement in $\C^2$. 
Our approach is to consider local system cohomology groups via 
$q$-deformations of the Aomoto complexes. 
Let $\omega=\sum_{i=1}^n a_ie_i\in A_\Z^1(\A)$. 
Let $q_0\in\C^\times\smallsetminus\{1\}$. 
Define the homomorphism $\rho(\omega, q_0): 
\pi_1(M(\A))\longmapsto\C^\times$ as 
\[
\rho(\omega, q_0)(\gamma_i)=q_0^{a_i}. 
\]
We denote the associated local system by 
$\scL_{\rho(\omega, q_0)}$. The main result asserts that 
there exists a $q$-deformation of the Aomoto complex whose 
specialization to $q=q_0$ computes the local system cohomology group. 

Let 
\begin{equation}
\label{eq:basis}
\begin{cases}
1 &\in A_\Z^0(\A), \\
\eta_1, \dots, \eta_n&\in A_\Z^1(\A),\\ 
\theta_1, \dots, \theta_b&\in A_\Z^2(\A)
\end{cases}
\end{equation}
be a $\Z$-basis of the Aomoto complex ($b=\rank A_\Z^2(\A)$). 
Let $\omega\in A_\Z^1(\A)$. The coboundary map of the 
Aomoto complex is expressed as 
\[
\begin{split}
\omega\wedge 1&=s_1\eta_1+\dots+s_n\eta_n\\
\omega\wedge \eta_i&=t_{1i}\theta_1+\dots+t_{bi}\theta_b, 
\end{split}
\]
with $s_i, t_{ij}\in\Z$. Denote the coefficient matrices by 
$S(\omega)=(s_i)$ and $T(\omega)=(t_{ij})$. 
Since they are matrix representations of coboundary maps, 
we have the following relation 
\begin{equation}
\label{eq:matrixrel}
T(\omega)S(\omega)=
\begin{pmatrix}
t_{11}&t_{12}&\dots&t_{1n}\\
t_{21}&t_{22}&\dots&t_{2n}\\
\vdots&\vdots&\ddots&\vdots\\
t_{b1}&t_{b2}&\dots&t_{bn}
\end{pmatrix}
\begin{pmatrix}
s_1\\ s_2\\ \vdots\\ s_n
\end{pmatrix}
=0. 
\end{equation}
Generally, 
the canonical $q$-deformation breaks the 
relation (\ref{eq:matrixrel}) (see Example \ref{ex:3lines}). 
\begin{definition}
\label{def:canoq}
The basis (\ref{eq:basis}) is called 
\emph{canonically $q$-deformable} if 
\begin{equation}
[T(\omega)]_q[S(\omega)]_q:=
\begin{pmatrix}
[t_{11}]_q&[t_{12}]_q&\dots&[t_{1n}]_q\\
[t_{21}]_q&[t_{22}]_q&\dots&[t_{2n}]_q\\
\vdots&\vdots&\ddots&\vdots\\
[t_{b1}]_q&[t_{b2}]_q&\dots&[t_{bn}]_q
\end{pmatrix}
\begin{pmatrix}
[s_1]_q\\ [s_2]_q\\ \vdots\\ [s_n]_q
\end{pmatrix}
=0. 
\end{equation}
Then, we denote the $q$-deformed Aomoto complex 
\begin{equation}
\begin{CD}
0
@>>>
A_{\CG}^0(\A)
@> [S(\omega)]_q>>
A_{\CG}^1(\A)
@>[T(\omega)]_q>>
A_{\CG}^2(\A)
@>>>
0
\end{CD}
\end{equation}
by $(A_{\CG}^\bullet(\A), \nabla_\omega^q)$. 
\end{definition}

\begin{example}
\label{ex:3lines}
Let $\A=\{H_1, H_2, H_3\}$ be three lines as in Figure \ref{fig:3lines}. 
Note that by the definition of Orlik-Solomon algebra, 
$e_1e_3=e_1e_2+e_2e_3$. 
Let $\omega=e_1+e_2+e_3$ and consider the $q$-deformation of 
the Aomoto complex $(A_\Z^\bullet(\A), \omega)$. 
\begin{figure}[htbp]
\centering
\begin{tikzpicture}[scale=0.8]
%格子
%\draw [help lines] (0,0) grid (12,6);%(0,0)から(10,4)までの"細線の方眼"
\draw[thick] (0,0) node[below] {$H_3$} --(4,2);
\draw[thick] (2,0) node[below] {$H_2$} --(2,2);
\draw[thick] (4,0) node[below] {$H_1$} --(0,2);
\end{tikzpicture}
\caption{$\A=\{H_1, H_2, H_3\}$} 
\label{fig:3lines}
\end{figure}
\begin{itemize}
\item[(1)] 
Consider the basis 
$1\in A_\Z^0(\A), e_1, e_2, e_3\in A_\Z^1(\A)$, and 
$e_1e_2, e_2e_3\in A_\Z^2(\A)$. Then the coefficients matrices are 
\[
T(\omega)=
\begin{pmatrix}
-2&1&1\\
-1&-1&2
\end{pmatrix}, 
\mbox{ and } 
S(\omega)=
\begin{pmatrix}
1\\1\\1
\end{pmatrix}. 
\]
Since $[2]_q\neq [1]_q+[1]_q$, we have 
$[T(\omega)]_q[S(\omega)]_q\neq 0$. 
Therefore, the basis is not canonically $q$-deformable. 
\item[(2)] 
Consider the basis 
$1\in A_\Z^0(\A), 
\eta_1=e_1, \eta_2=e_1-e_2, \eta_3=e_2-e_3\in A_\Z^1(\A)$, and 
$e_1e_2, e_2e_3\in A_\Z^2(\A)$. Then the coefficients matrices are 
\[
T(\omega)=
\begin{pmatrix}
-2&-3&0\\
-1&0&-3
\end{pmatrix}, 
\mbox{ and } 
S(\omega)=
\begin{pmatrix}
3\\ -2\\ -1
\end{pmatrix}. 
\]
In this case, we have $[T(\omega)]_q[S(\omega)]_q=0$. 
Actually, the basis is 
a canonically $q$-deformable basis. 
\end{itemize}
\end{example}
Thus the canonical $q$-deformability depends on the 
choice of the basis. We also make the following definition. 
\begin{definition}
\label{def:univ}
A canonically $q$-deformable basis (\ref{eq:basis}) is called 
\emph{universal} if the $q$-deformation computes the local system 
cohomology group. Namely, for any $\omega\in A_\Z^1(\A)$ and 
$q_0\in\C^\times\smallsetminus\{1\}$, the cohomology 
\[
\frac{\Ker([T(\omega)]_{q_0}:A_\C^1(\A)\longrightarrow A_\C^2(\A))}
{\Im([S(\omega)]_{q_0}:A_\C^0(\A)\longrightarrow A_\C^1(\A))}
\]
is isomorphic to $H^1(M(\A), \scL_{\rho(\omega, q_0)})$. 
\end{definition}
Our main result is the following. 
\begin{theorem}
\label{thm:realmain}
Let $\A=\{H_1, \dots, H_n\}$ be a complexified real 
line arrangement. Then, there exists 
a universal canonically $q$-deformable basis of 
the Aomoto complex 
$1 \in A_\Z^0(\A), \eta_1, \dots, \eta_n\in A_\Z^1(\A),  
\theta_1, \dots, \theta_b\in A_\Z^2(\A)$. 
\end{theorem}

%\begin{remark}
%The 
%\end{remark}

\subsection{Proof of Theorem \ref{thm:realmain}}

\label{subsec:proof}

The proof is based on the previous results 
\cite{yos-ch, yos-loc, yos-lef}. 
We first recall the a general result 
\cite[Theorem 4.1]{yos-ch} concerning 
the relation between the Aomoto complex and the twisted cohomology 
group for complexified real arrangements. 

Let $\A=\{H_1, \dots, H_n\}$ be an affine hyperplane arrangement 
in $\R^\ell$. Let $M(\A)$ be the complexified complement 
$M(\A)=\C^\ell\smallsetminus\bigcup_{H\in\A}H\otimes\C$. 
We describe the twisted cohomology $H^k(M(\A), \scL)$ by using 
chambers. 

Recall that a connected component of $\R^\ell\smallsetminus
\bigcup_{H\in\A}H$ is called a \emph{chamber}. Denote by 
$\ch(\A)$ the set of all chambers. Let 
\[
\scF:\ F^0\subset F^1\subset\dots\subset F^\ell=\R^\ell
\]
be a generic flag, that is, $F^k$ is a generic $k$-dimensional 
affine subspace. For a given generic flag $\scF$, we decompose 
the set of chambers. Let 
\[
\ch_{\scF}^k(\A)=\{C\in\ch(\A)\mid C\cap F^k\neq\emptyset, 
C\cap F^{k-1}=\emptyset\}, 
\]
for $0\leq k\leq \ell$, where $F^{-1}=\emptyset$. Then the 
set $\ch(\A)$ of chambers is decomposed into 
\[
\ch(\A)=\ch_{\scF}^0(\A)\sqcup\ch_{\scF}^1(\A)\sqcup\dots
\sqcup\ch_{\scF}^\ell(\A). 
\]
%It is known that $\#\ch_{\scF}^k(\A)=b_k(M(\A))$ 
%(\cite[Proposition 2.3.2]{yos-lef}). 
It is known that $\#\ch_{\scF}^k(\A)=\rank_{\Z}A_{\Z}^k(\A)$ 
\cite[Proposition 2.3.2]{yos-lef}. 
Furthermore, $\ch_{\scF}^k(\A)=\{C_1^k, \dots, C_{b_k}^k\}$ 
naturally determines a basis $[C_1^k], \dots, [C_{b_k}^k]\in A_\Z^k(\A)$ 
\cite[Theorem 3.2]{yos-ch}, which is called a 
\emph{chamber basis}. 

Let $\omega\in A_\Z^1(\A)$. We can express 
the coboundary map of the Aomoto complex using 
the chamber basis as 
\begin{equation}
\label{eq:aomchamber}
\omega\wedge[C^k_i]=\sum_{j}\Gamma_{ij}\cdot [C^{k+1}_j], 
\end{equation}
with $\Gamma_{ij}\in\Z$. 
Then \cite[Theorem 4.1]{yos-ch} asserts that 
there is a decomposition 
\begin{equation}
\label{eq:decompNL}
\Gamma_{ij}=N_{ij}\cdot L_{ij}, 
\end{equation}
with $N_{ij}, L_{ij}\in\Z$ such that the map 
\begin{equation}
\nabla_\omega: 
A_{\CG}^k(\A)\longrightarrow A_{\CG}^{k+1}(\A), \ 
[C^k_i]\longmapsto
\sum_{j}N_{ij}\cdot (q^{L_{ij}/2}-q^{-L_{ij}/2})\cdot [C^{k+1}_j]
\end{equation}
defines a cochain complex $(A_{\CG}^\bullet(\A), \nabla_\omega)$ 
whose cohomology group at 
$q=q_0$ is 
isomorphic to $H^k(M(\A), \scL_{\rho(\omega, q_0)})$ for 
any $q_0\in\C^\times\smallsetminus\{1\}$. 
Replace the chamber basis by 
$\eta^k_i=(q^{1/2}-q^{-1/2})^k[C^k_i]$. The map $\nabla_\omega$ 
is expressed as 
\begin{equation}
\nabla_\omega(\eta^k_i)=
\sum_{j}N_{ij}\cdot [L_{ij}]_q\cdot \eta^{k+1}_j, 
\end{equation}
which can be considered as a $q$-deformation of the 
coboundary map of the Aomoto complex 
(\ref{eq:aomchamber}). 

Now we consider the case $\ell=2$. The integer $N_{ij}$ is 
equal to (up to sign) the degree map introduced in 
\cite[Definition 6.3.3]{yos-lef}. When $\ell=2$, the value is 
known to be $N_{ij}\in\{0, \pm 1\}$ \cite[Definition 3.2 (1)]{yos-loc}. 
Thus we have 
\begin{equation}
\label{eq:qscalar}
N_{ij}\cdot [L_{ij}]_q=[N_{ij}\cdot L_{ij}]_q= [\Gamma_{ij}]_q, 
\end{equation}
which is nothing but the canonical $q$-deformation of the 
coboundary map of the Aomoto complex with respect to the 
chamber basis. This completes the proof of 
Theorem \ref{thm:realmain}. 

\begin{remark}
By the proof, the Aomoto complex $(A_\Z^\bullet(\A), \omega)$ 
has a $q$-deformation which has the universality in the sense 
of Definition \ref{def:univ}. However, it is not necessarily the 
canonical $q$-deformation with respect to certain basis. 
The reason is that if $\ell\geq 3$, there are examples such that 
$N_{ij}\leq -2$ (actually, it can take any value in 
$\{1,0.-1, -2, -3, \dots\}$). Thus (\ref{eq:qscalar}) does not hold 
in general. 
\end{remark}

%\begin{remark}
%Theorem \ref{thm:realmain} shows that 
%\end{remark}

\section{Examples and remarks}

\label{sec:ex}

\subsection{Deleted $B_3$-arrangement}

\label{subsec:delB3}

In this section, we exhibit some computations of $q$-deformations 
of the Aomoto complex with respect to the chamber basis for 
the deleted $B_3$-arrangement. 
The deleted $B_3$-arrangement was studied in detail by 
Suciu \cite{suc-fundam, suc-trans} as the first example which 
has a positive dimensional translated component in 
the characteristic variety. 

Let $\A=\{H_1, \dots, H_7\}$ be 
the line arrangement as in Figure \ref{fig:delB3}. 
Let 
\begin{equation}
\omega=e_1+2e_2+2e_3+e_4+2e_5+3e_6+4e_7\in A_\Z^1(\A). 
\end{equation}

\begin{figure}[htbp]
\centering
\begin{tikzpicture}[scale=0.8]

%格子
%\draw [help lines] (0,0) grid (12,6);%(0,0)から(10,4)までの"細線の方眼"

\draw[thick, black] (0,0) -- +(8,4) node[right]{$H_3$}; 
\draw[thick, black] (0,2) -- +(8,4) node[right]{$H_4$}; 
\draw[thick, black] (0,4) -- +(8,-4) node[right]{$H_1$}; 
\draw[thick, black] (0,6) -- +(8,-4) node[right]{$H_2$}; 
\draw[thick, black] (2,0) -- +(0,6) node[above]{$H_7$}; 
\draw[thick, black] (4,0) -- +(0,6) node[above]{$H_6$}; 
\draw[thick, black] (6,0) -- +(0,6) node[above]{$H_5$}; 

\end{tikzpicture}
\caption{The deleted $B_3$-arrangement} 
\label{fig:delB3}
\end{figure}
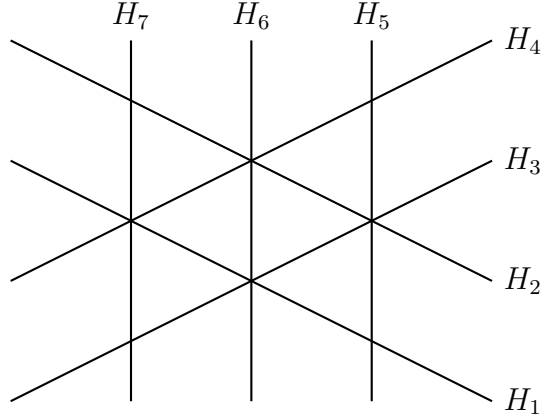

Take the generic flag $F^0\subset F^1$ as in Figure \ref{fig:flag}. 
Then the set of chambers is decomposed as 
\[
\begin{split}
\ch^0(\A)&=\{C_0\}\\
\ch^1(\A)&=\{C_1, C_2, \dots, C_7\}\\
\ch^2(\A)&=\{D_1, D_2, \dots, D_{12}\}
\end{split}
\]

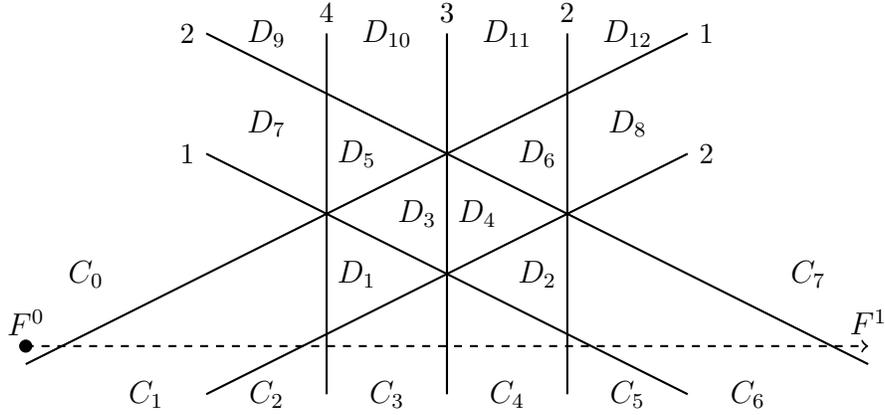
\begin{figure}[htbp]
\centering
\begin{tikzpicture}[scale=0.8]

%格子
%\draw [help lines] (-3,0) grid (12,6);%(0,0)から(10,4)までの"細線の方眼"

\draw[thick, black] (0,0) -- +(8,4) node[right] {\small $2$}; 
\draw[thick, black] (-3,0.5) -- +(11,5.5) node[right] {\small $1$}; 
\draw[thick, black] (0,4) node[left] {\small $1$} -- +(8,-4); 
\draw[thick, black] (0,6) node[left] {\small $2$} -- +(11,-5.5); 
\draw[thick, black] (2,0) -- +(0,6) node[above] {\small $4$}; 
\draw[thick, black] (4,0) -- +(0,6) node[above] {\small $3$}; 
\draw[thick, black] (6,0) -- +(0,6) node[above] {\small $2$}; 

\filldraw[fill=black, draw=black] (-3,0.8) node[above] {$F^0$} circle (0.1);
\draw[thick, dashed, ->] (-3,0.8) -- (11,0.8) node[above] {$F^1$} ;

\draw (-2,2) node {$C_0$}; 
\draw (-1,0) node {$C_1$}; 
\draw (1,0) node {$C_2$}; 
\draw (3,0) node {$C_3$}; 
\draw (5,0) node {$C_4$}; 
\draw (7,0) node {$C_5$}; 
\draw (9,0) node {$C_6$}; 
\draw (10,2) node {$C_7$}; 

\draw (2.5,2) node {$D_1$}; 
\draw (5.5,2) node {$D_2$}; 
\draw (3.5,3) node {$D_3$}; 
\draw (4.5,3) node {$D_4$}; 
\draw (2.5,4) node {$D_5$}; 
\draw (5.5,4) node {$D_6$}; 
\draw (1,4.5) node {$D_7$}; 
\draw (7,4.5) node {$D_8$}; 
\draw (1,6) node {$D_9$}; 
\draw (3,6) node {$D_{10}$}; 
\draw (5,6) node {$D_{11}$}; 
\draw (7,6) node {$D_{12}$}; 

\end{tikzpicture}
\caption{Generic flag and chambers (the small numbers adjacent to 
lines are coefficients of $\omega$).} 
\label{fig:flag}
\end{figure}

We describe the Aomoto complex $(A_\Z^\bullet(\A), \omega)$ 
with respect to the chamber basis 
$[C_0]=1\in A_\Z^0(\A), [C_1], \dots, [C_7]\in A_\Z^1(\A), 
[D_1], \dots [D_{12}]\in A_\Z^2(\A)$. First, we have 
\[
\omega\cdot 1=[C_1]+3[C_2]+7[C_3]+10[C_4]+12[C_5]+13[C_6]+15[C_7].
\]
Next, the coefficient of $\omega\wedge[C_j]=\sum_{i=1}^{12}c_{ij}[D_i]$ 
is computed as follows (\cite{yos-ch, yos-loc}). 
\begin{equation}
\begin{array}{c||c|c|c|c|c|c|c|}
&C_1&C_2&C_3&C_4&C_5&C_6&C_7\\
\hline
\hline
D_1&-4&6&-2&&&&\\
\hline
D_2&&&&-1&3&-2&\\
\hline
D_3&-5&7&-3&&8&-7&\\
\hline
D_4&-8&10&&-3&5&-4&\\
\hline
D_5&&8&-4&&9&-8&\\
\hline
D_6&-10&12&&-5&7&&-4\\
\hline
D_7&&&&&13&-12&\\
\hline
D_8&-12&14&&&&&-2\\
\hline
D_9&&&&&15&&-12\\
\hline
D_{10}&&10&-6&&11&&-8\\
\hline
D_{11}&&13&&-6&8&&-5\\
\hline
D_{12}&&15&&&&&-3\\
\hline
\end{array}
\end{equation}
By Theorem \ref{thm:realmain}, this basis is univeral and 
canonically $q$-deformable. Hence, we have the following 
relation (we can also easily check by using 
Clebsch-Gordan relation (\ref{eq:cg})), 
\begin{equation}
\label{eq:qmatrix}
\begin{pmatrix}
-[4]_q&[6]_q&-[2]_q&0&0&0&0\\
0&0&0&-[1]_q&[3]_q&-[2]_q&0\\
-[5]_q&[7]_q&-[3]_q&0&[8]_q&-[7]_q&0\\
-[8]_q&[10]_q&0&-[3]_q&[5]_q&-[4]_q&0\\
0&[8]_q&-[4]_q&0&[9]_q&-[8]_q&0\\
-[10]_q&[12]_q&0&-[5]_q&[7]_q&0&-[4]_q\\
0&0&0&0&[13]_q&-[12]_q&0\\
-[12]_q&[14]_q&0&0&0&0&-[2]_q\\
0&0&0&0&[15]_q&0&-[12]_q\\
0&[10]_q&-[6]_q&0&[11]_q&0&-[8]_q\\
0&[13]_q&0&-[6]_q&[8]_q&0&-[5]_q\\
0&[15]_q&0&0&0&0&-[3]_q
\end{pmatrix}
\begin{pmatrix}
[1]_q\\
[3]_q\\
[7]_q\\
[10]_q\\
[12]_q\\
[13]_q\\
[15]_q
\end{pmatrix}
=0, 
\end{equation}
and any specialization $q\in\C^\times\smallsetminus\{1\}$ gives 
the local system cohomology $H^1(M(\A), \scL_{\rho(\omega, q)})$. 
Here, we consider two cases, $q=-1$ and $q=\exp(\pi\sqrt{-1}/3)$. 

\subsection{The case $q=-1$}

\label{subsec:minus1}

In this case, $q^{1/2}=(-1)^{1/2}=\sqrt{-1}$, and $q^{-1/2}=-\sqrt{-1}$, 
and $[n]_{-1}$ is as follows. 
\begin{equation}
%\frac{1}{\sqrt{-1}}
[n]_{-1}=
\begin{cases}
1& \mbox{ if $n\equiv 1 \mod 4$,}\\
-1& \mbox{ if $n\equiv 3 \mod 4$}\\
0& \mbox{ if $n\equiv 0, 2 \mod 4$.}\\
\end{cases}
\end{equation}
The coefficient matrix (\ref{eq:qmatrix}) is specialized to 
\begin{equation}
%\label{eq:qmatrix}
\begin{pmatrix}
0&0&0&0&0&0&0\\
0&0&0&-1&-1&0&0\\
-1&-1&1&0&0&1&0\\
0&0&0&1&1&0&0\\
0&0&0&0&1&0&0\\
0&0&0&-1&-1&0&0\\
0&0&0&0&1&0&0\\
0&0&0&0&0&0&0\\
0&0&0&0&-1&0&0\\
0&0&0&0&-1&0&0\\
0&1&0&0&0&0&-1\\
0&-1&0&0&0&0&1
\end{pmatrix}
\end{equation}
This matrix has rank $4$. 
Hence $\dim H^1(M(\A), \scL_{\rho(\omega, -1)})=2$. 
Indeed, the character $(q_1, \dots, q_7)=(-1, 1, 1, -1, 1, -1, 1)$ 
is one of the two isolated points at which the twisted cohomology 
has $\dim H^1(M(\A), \scL)=2$ (\cite[Example 10.6]{suc-fundam}).

\subsection{The case $q=\exp(\pi\sqrt{-1}/3)$} 

\label{subsec:6root1}

Let $q=\zeta:=\exp(\pi\sqrt{-1}/3)$ be a primitive $6$-th root of $1$. 
In this case, $q^{1/2}=\exp(\pi\sqrt{-1}/6)$, 
$q^{-1/2}=\exp(-\pi\sqrt{-1}/6)$, and we have 
\begin{equation}
%\frac{1}{\sqrt{-1}}
[n]_{\zeta}=
\begin{cases}
0& \mbox{ if $n\equiv 0, 6 \mod 12$,}\\
1& \mbox{ if $n\equiv 1, 5 \mod 12$,}\\
\sqrt{3}& \mbox{ if $n\equiv 2, 4 \mod 12$,}\\
2& \mbox{ if $n\equiv 3 \mod 12$,}\\
-1& \mbox{ if $n\equiv 7, 11 \mod 12$,}\\
-\sqrt{3}& \mbox{ if $n\equiv 8, 10 \mod 12$,}\\
-2& \mbox{ if $n\equiv 9 \mod 12$.}
\end{cases}
\end{equation}
The coefficient matrix (\ref{eq:qmatrix}) is specialized to 
\begin{equation}
\begin{pmatrix}
-\sqrt{3}&0&-\sqrt{3}&0&0&0&0\\
0&0&0&-1&2&-\sqrt{3}&0\\
-1&-1&-2&0&-\sqrt{3}&1&0\\
\sqrt{3}&-\sqrt{3}&0&-2&1&-\sqrt{3}&0\\
0&-\sqrt{3}&-\sqrt{3}&0&-2&-\sqrt{3}&0\\
\sqrt{3}&0&0&-1&-1&0&-\sqrt{3}\\
0&0&0&0&1&0&0\\
0&\sqrt{3}&0&0&0&0&-\sqrt{3}\\
0&0&0&0&2&0&0\\
0&-\sqrt{3}&0&0&-1&0&\sqrt{3}\\
0&1&0&0&-\sqrt{3}&0&-1\\
0&2&0&0&0&0&-2
\end{pmatrix}
\end{equation}
This matrix has rank $5$. Hence 
$\dim H^1(M(\A), \scL_{\rho(\omega, \zeta)})=1$. 
Indeed, the character 
$(q_1, \dots, q_7)=
(\zeta, \zeta^2, \zeta^2, \zeta, \zeta^2, \zeta^3, \zeta^4)$ 
(where $\zeta=\exp(\pi\sqrt{-1}/3)$) is a point in the translated 
component 
\[
\Omega=\{(t, -t^{-1}, -t^{-1}, t, t^2, -1, t^{-2})\mid t\in\C^\times\} 
\]
described in \cite[Example 10.6]{suc-fundam}.

\subsection{Concluding remarks}
\label{subsec:rem}

Once a universal canonically $q$-deformable basis of 
the Aomoto complex is obtained, the local system cohomology 
group $H^k(M(\A), \scL)$ is computed combinatorially 
(for rank one local systems expressed as 
$\scL=\scL_{\rho(\omega, q_0)}$). However, 
at this moment, the utilization of real structures (chambers) 
is necessary to select such a nice basis. 
Naturally, the following problems emerge. 

\begin{problem}
\begin{itemize}
\item[(1)] 
Can one characterize universal canonically $q$-deformable 
bases of the Aomoto complex combinatorially for complexfied 
real line arrangements? 
\item[(2)] 
Do such bases exist for higher-dimensional complex 
hyperplane arrangements? 
\end{itemize}
\end{problem}

A universal canonically $q$-deformable basis of the Aomoto complex 
requires highly non-trivial relations among $q$-integers such as 
(\ref{eq:qmatrix}). However, as we saw in \S \ref{sec:q}, 
all the relations among $q$-integers 
are derived from the Clebsch-Gordan relations 
(Proposition \ref{prop:qrel}) which 
has an origin in the irreducible decomposition of the 
tensor product of highest weight representations of $\Uqsl$ 
(Remark \ref{rem:cg}). 
\begin{problem}
Can we categorify the $q$-deformation of the Aomoto complex 
in a similar manner that (\ref{eq:cgdec}) categorifies (\ref{eq:cg})? 
\end{problem}

Another interesting topic related to $q$-deformation is the 
notion of 
\emph{cyclic sieving phenomena} (CSP). Many enumerative 
problems are naturally equipped with a cyclic group action. 
The number of objects which are invariant under 
a cyclic subgroup is sometimes computed from 
the $q$-deformation of the original enumerative formula as 
the special value at a root of $1$. 
See \cite{csp} for more details on CSP. 
Universal canonically $q$-deformable bases of the Aomoto 
complex enable us to compute the monodromy eigenspaces 
of the cohomology of a cyclic covering space of $M(\A)$ by 
specializing the $q$-deformation at a root of $1$. 
There are notable parallel structures between CSP and the 
framework of this paper. 
Exploring the relationship between CSP and the topology 
of cyclic covering spaces of arrangements complements 
promises to be an interesting research direction.

\end{document}